\documentclass[a4paper, 10pt, twoside, notitlepage]{amsart}

\usepackage{amsmath,amscd}
\usepackage{amssymb}
\usepackage{amsthm}
\usepackage{comment}
\usepackage{graphicx}

\usepackage{mathrsfs}
\usepackage[ocgcolorlinks, linkcolor=blue]{hyperref}

\usepackage[ocgcolorlinks,linkcolor=blue]{hyperref}
\theoremstyle{plain}
\newtheorem{thm}{Theorem}[section]
\newtheorem{prop}{Proposition}[section]
\newtheorem{rem}{Remark}[section]
\newtheorem{lem}[prop]{Lemma}

\newtheorem{defi}[prop]{Definition}
\newtheorem{rmk}[prop]{Remark}

\numberwithin{equation}{section}
\newcommand {\R} {\mathbb{R}} \newcommand {\Z} {\mathbb{Z}}
 \newcommand {\N} {\mathbb{N}}
 
\newcommand {\p} {\partial}

\newcommand{\eps}{\epsilon}

\newcommand{\wt}{\widetilde}
\newcommand{\HH}{\mathcal{H}}

\newcommand{\norm}[1]{\lVert #1 \rVert}         % Formatting for the norm
\DeclareMathOperator{\G}{\mathcal{G}}

\DeclareMathOperator {\dist} {dist}

\DeclareMathOperator{\diam}{diam}

\renewcommand{\theequation}{\thesection.\arabic{equation}}
\title[Duality between RT and NRT]{Duality between range and no-response tests and its application for inverse problemS}

\author[Lin]{Yi-Hsuan Lin}
\address{Department of Applied Mathematics, National Chiao Tung University, Hsinchu 30010, Taiwan. }
\curraddr{}
\email{yihsuanlin3@gmail.com}

\author[Nakamura]{Gen Nakamura}
\address{Department of Mathematics, Hokkaido University, Sapporo 060-0810, Japan}
\curraddr{}
\email{nakamuragenn@gmail.com}

\author[Potthast]{Roland Potthast}
\address{Department of Mathematics and Statistics,
	University of Reading, RG6 6AH, UK}
\curraddr{}
\email{r.w.e.potthast@reading.ac.uk}

\author[Wang]{Haibing Wang}
\address{School of Mathematics, Southeast University, Nanjing 210096, China}
\curraddr{}
\email{hbwang@seu.edu.cn}

\begin{document}
	
	\maketitle
	
	\begin{abstract}
		In this paper we will show the duality between the range test (RT) and no-response test (NRT) for the inverse boundary value problem for the Laplace equation in $\Omega\setminus\overline D$ with
		an obstacle $D\Subset\Omega$ whose boundary $\partial D$ is visible from the boundary $\partial\Omega$ of $\Omega$ and a measurement is given as a set of Cauchy data on $\partial\Omega$. Here the Cauchy data is given by a unique solution $u$ of the boundary value problem for the Laplace equation in $\Omega\setminus\overline D$ with homogeneous and inhomogeneous Dirichlet boundary condition on $\partial D$ and $\partial\Omega$, respectively. These testing methods are domain sampling method to estimate the location of the obstacle using test domains and the associated indicator functions. Also both of these testing methods can test the analytic extension of $u$ to the exterior of a test domain. Since these methods are defined via some operators which are dual to each other, we could expect that there is a duality between the two methods. We will give this duality in terms of the equivalence of the pre-indicator functions associated to their indicator functions. As an application of the duality, the reconstruction of $D$ using the RT gives the reconstruction of $D$ using the NRT and vice versa. We will also give each of these reconstructions without using the duality if the Dirichlet data of the Cauchy data on $\partial\Omega$ is not identically zero and the solution to the associated forward problem does not have any analytic extension across $\partial D$.  
		Moreover, we will show that these methods can still give the reconstruction of $D$ if $D$ is a convex polygon and it satisfies one of the following two properties: all of its corner angles are irrational and its diameter is less than its distance to $\partial\Omega$. 		
		\medskip
		
		\noindent{\bf Keywords.}  Inverse boundary value problem,  range test, no-response test, duality
		
		\noindent{\bf Mathematics Subject Classification (2010)}:  31B20, 35J15, 65N21
		
	\end{abstract}

	%	\tableofcontents

	\section{Introduction}
	\setcounter{equation}{0}
	 We will first set up our inverse problem. To begin with let $\Omega \subset \R^n$ for $n=2,3$ be a bounded domain with $C^{2}$ boundary $\partial\Omega$. Physically $\Omega$ is a medium and it can be either homogeneous electric or heat conductive medium with conductivity $1$. Let $D\Subset \Omega$ be a domain with Lipschitz boundary $\partial D$ such that $\Omega\setminus \overline{D}$ is connected. Then the voltage or temperature of electric or heat denoted by $u$ satisfies the following boundary value problem 
	\begin{align}\label{Main equation}
	\begin{cases}
	\Delta u =0 & \text{ in }\Omega \setminus \overline{D}, \\
	u=0 & \text{ on }\p D, \\
	u=f & \text{ on }\p \Omega,
	\end{cases}
	\end{align}
	where $f$ is taken from the $L^2$ based Sobolev space $H^{1/2}(\partial\Omega)$ of order $1/2$ on $\partial\Omega$ which is a specified voltage or temperature at $\partial\Omega$. The physical meaning of the Dirichlet boundary condition for $\partial D$ is the earthing boundary for an electric conductive medium and the prescribed $0$ temperature for a heat conductive medium. 
	
	It is well known that \eqref{Main equation} is well-posed. That is for any given $f\in H^{1/2}(\partial\Omega)$ of order $1/2$ on $\partial\Omega$, there exists a unique solution $u=u_f$ in the $L^2$ based Sobolev space $ H^1(\Omega\setminus\overline D)$ of order $1$ in $\Omega\setminus\overline D$ to \eqref{Main equation} such that
	\begin{equation*}\label{well-posedness}
	\Vert u\Vert_{H^1(\Omega\setminus\overline D)}\le C \Vert f\Vert_{H^{1/2}(\partial\Omega)}
	\end{equation*}
	for some constant $C>0$ which does not depend on $f$ and $u$. Henceforth we call such a $C>0$ general constant, which may differ from place to place, but we will use the same notation $C$. 
	%Sometimes we will suppress the general constant $C$ and just write $\le C$ as $\lesssim$.  

	Based on this well-posedness, one can calculate the Neumann derivative $\p_\nu u_f$ on $\partial\Omega$ which belongs to the dual space $H^{-1/2}(\partial\Omega)$ of $H^{1/2}(\partial\Omega)$, and this means that we can measure either electric current or heat flux on  $\partial\Omega$. The pair $\left\{f,\left. \partial_\nu u_f\right|_{\partial\Omega}\right\}$ with the unit normal $\nu$ of $\partial\Omega$ directed outside $\Omega$ is called a Cauchy data. Throughout this paper, we assume that the boundary data $f$ on $\p \Omega$ is a \emph{non-identically zero} function, and $u\in H^1(\Omega\setminus\overline D)$ always stands for the solution to \eqref{Main equation}.  
	Then our inverse boundary value problem can be stated as follows.
	
	\medskip
	
	\noindent{\bf Inverse Boundary Value Problem}
	
	\vspace{1mm}
	\noindent Given a set of Cauchy data $\left\{f|_{\p \Omega},\left. \partial_\nu u_f\right|_{\partial\Omega}\right\}$ taken as our measurement, identify $D$ from this measurement. 
	
	\medskip
	The uniqueness of this inverse problem has been already known very early for example from the proof given for the uniqueness of identifying an unknown rigid inclusion inside an isotropic elasticity medium \cite{Ang}. Also the stability estimate for the identification is known for the conductivity equation \cite{Alessandrini} and even for the isotropic elasticity system \cite{Higashimori,Morassi}. The next natural problem is to give a reconstruction method to reconstruct $D$ from the given Cauchy data. We are particularly interested in the reconstruction methods called the range test (RT) and the no response test (NRT) very well known in the inverse scattering problem. 
	The RT and NRT were introduced by Potthast-Sylvester-Kusiak in \cite{KPS} and Luke-Potthast in \cite{Luke} both for the inverse acoustic scattering problem to identify a scatterer such as a sound soft or sound hard obstacle, respectively. There are single wave RT/NRT and multiple waves RT/NRT. The corresponding measurements are the far field of the scattered wave generated by one incident plane wave and the scattering amplitude generated by multiple incident waves, respectively. Here it should be remarked that the multiple incident waves mean infinitely many incident waves. The multiple waves RT/NRT can recover the scatterer, but the single wave RT/NRT in general can only give the upper estimate of the convex scattering support. Moreover the single wave RT gives an easy test for analytic extension of $u$ (see  (\cite{KPS}).  For further information about the RT/NRT for the inverse acoustic scattering problems see \cite{Nakamura}, \cite{Luke} and \cite{Potthast}. RT/NRT were also applied to inverse scattering problems for electromagnetic waves (\cite{PS}) and the Oseen flow linearized around a constant velocity field (\cite{ZP}). We will refer the single wave RT and NRT adapted to our inverse problem by RT and NRT which will be given in Section \ref{Section 1} and Section \ref{duality}, respectively. There is a duality known between the RT and NRT for the inverse scattering problem for acoustic waves (\cite{Nakamura}) and for the Oseen flow linearized at constant velocity field (\cite{ZP}). The RT/NRT for the inverse boundary value problems become a bit more complicated than those for the inverse scattering problems. The aims of this paper are to give the corresponding duality for our inverse problem, and by using either the RT or the NRT to test the analytic extension of the solution $u$ of \eqref{Main equation}, reconstruct $D$ under the assumption that $u$ does not have analytic extension across $\partial D$.
	Our main results are the following three theorems. 
	\begin{thm}\label{main result} There is a duality between the RT and NRT for the inverse boundary value problem.	Its details will be given in Section \ref{duality}.
	\end{thm}
	
	%It is possible to obtain a result similar to Theorem \ref{main result} for more general equations such as the conductivity equation with anisotropic and heterogeneous conductivity, and also for the static elasticity equation with isotropic and heterogeneous elasticity tensor. 
	The same is true for an unknown $D$	with Neumann boundary condition at $\partial D$. 
	
	\medskip
	As an application of testing the analytic extension of $u$ and the above duality theorem, we have the following reconstruction of $D$.
	\begin{thm}\label{reconstruction}
		Consider the inverse boundary value problem. Then either using the RT or NRT, we can reconstruct $D$ from the above given Cauchy data if $u$ does not have any analytic extension across $\partial D$. Their details
		will be given in Section \ref{rec by RT} for the RT and Section \ref{rec by NRT} for the NRT, respectively. 
		
	\end{thm}
	
	It can be noticed from these two theorems and a short introduction on the RT/NRT given before the theorems that the key behind the RT/NRT is the analyticity. These theorems can be further generalized to equations with analytic coefficients and the associated Green functions have the jump property, the decaying property at infinity, the mapping property and the Fredholm of index zero property for the trace of double layer potentials.
	
	A typical example that $u$ does not have any analytic extension across $\partial D$ is the case $\partial D$ is not analytic everywhere (see Corollary 2.3 of \cite{HNS}).
	This assumption on $u$ is a very strong assumption. We can relax this assumption if $D$ is a convex polygon which has either irrationally angled corners (see Section \ref{sec convex case} for its definition) or the distance property (see the next theorem for its definition). More precisely we have the following theorem.
	
	\begin{thm}\label{convex}
		Consider the inverse boundary value problem. Let $D$ be a convex polygon. Assume that all of its corners have irrational angle or it  satisfies the \emph{distance property} given as 
		\begin{equation}\label{distance property}
		\diam D<\dist(D,\partial\Omega),
		\end{equation}
		where $\diam(D)$ and $\dist(D,\partial\Omega)$ denote the diameter of $D$ and the distance between $D$ and $\partial\Omega$, respectively. Further we assume the following \emph{existence of a priori convex polygon} given as follows. There exists a convex polygon $D_0$ such that $D\subset D_0\Subset\Omega$ and $\Omega\setminus\overline{D_0}$ is connected. Then either using the RT or NRT, we can reconstruct $D$ from the above given Cauchy data.
	\end{thm}
	
	\begin{rem}\label{Ikehata}${}$
		%\newline
		\begin{itemize}
			\item [{\rm (i)}] The existence of a priori convex polygon $D_0$ is necessary only in the case the NRT itself is used for the reconstruction. If we use the NRT with the duality between RT and NRT, we don't need to have this $D_0$. 
			\item[{\rm(ii)}] Replace the Dirichlet condition in \eqref{Main equation} by the Neumann condition, and assume that $u$ is not a constant function. Further let $D$ be a convex polygon which satisfies the distance property. Then Ikehata (\cite{Ikehata convex}) gave a reconstruction of $D$ from one set of Cauchy data by using his enclosure method.
		\end{itemize}
	\end{rem}
	
	Since there is a huge literature on the reconstruction methods for our inverse problem, we only give some major reconstruction methods by citing one paper for each method which we came across with strong interest. So we ask the readers to consult the literature there in and make further search
	to collect more information about the methods. They are iterative methods using the domain derivative \cite{Kress}, topological derivative \cite{Bonnet}, level set \cite{Burger} and quasi-reversibility \cite{Bourgeois}. Also there are non-iterative methods using the polarization tensor \cite{Ammari}, a family of special solutions called the complex geometrical optics solutions as test functions \cite{ikehata}, nearly orthogonal probing functions \cite{Zou} and point source \cite{HNS}. The second non-iterative method is called the enclosure method which we already mentioned in Remark \ref{Ikehata}. We remark that this last method was given for the inverse scattering problem for the Helmholtz equation, but it can be adapted to the inverse boundary value problem for the Laplace equation to give a reconstruction of an unknown obstacle with non-analytic boundary.
	
	In the rest of this paper except Section \ref{sec convex case} and Appendix we only consider the three dimensional case i.e. $n=3$ for the simplicity of description. Also we organize the rest of this paper as follows.
	In Section \ref{Section 1} we show that for a domain $G\Subset\Omega$ likewise $D$, the analytic extension of $u$ to $\Omega\setminus\overline G$ can be characterized by the RT. Such a domain will be called a test domain for RT and NRT. Then in Section \ref{duality} we will introduce the NRT and prove the duality between the RT and NRT. Based on the characterization of analytic extension of $u$ in terms of the RT, we will reconstruct $D$ from the given Cauchy data in Section \ref{rec by RT} by using the RT if the solution $u$ to \eqref{Main equation} does not have any analytic extension across $\partial D$. Further in Section \ref{rec by NRT}, we will reconstruct $D$ by using the NRT under the same situation as for the RT. 
	For these two reconstructions we have to emphasize that we will not use the duality.
	In the last section we will pick up the case that $D$ is a convex polygon and prove Theorem \ref{convex}. Appendix gives some propositions and their proofs cited in Sections \ref{Section 1} and \ref{sec convex case}.
	
	%%%%%%%%%%%%%%%%%%%%%	

	%%%%%%%%%%%%%%%%%%	
	\section{analytic extension and RT}\label{Section 1}
	\setcounter{equation}{0}
	
	%The inverse scattering is a theory to recover information about a medium and its embedded obstacles via the excited measurements.
	We first consider a version of the range test for the inverse boundary value problem for the Laplace equation. As we mentioned before, the RT was introduced in \cite{KPS} for the inverse scattering problem for the Helmholtz equation.
	
	For $f\in H^{1/2}(\partial\Omega)$, we recall that $u=u^f\in H^1(\Omega\setminus\overline D)$ is the solution to \eqref{Main equation}. Any domain $G\Subset\Omega$ likewise $D$ is called a \emph{test domain}. For a test domain $G$, we want to characterize the analytic extension of $u$ to $\Omega\setminus\overline G$ by using the RT, and by using this characterization, we want to analyze the analytic extension of $u$ across $\partial D$. The definition of analytic extension of $u$ across $\partial D$ is defined as follows.
	\begin{defi}\label{analytic extension}
		Let $\Omega$ and $D$ be open bounded sets given as before. A harmonic function $v\in H^1(\Omega\setminus\overline D)$ is said to have an \emph{analytic extension} across $\p D$ if and only if there exists a domain $D'$ with Lipschitz boundary $\partial D'$ such that 
		\[
		\overline{D'} \subsetneq \overline{D},  \qquad \mathcal{H}^{n-1}(\p D \cap \p D') >0,
		\]
		and a function $v'\in H^1(\Omega \setminus \overline{D'})$ such that 
		\begin{align*}
		\begin{cases}
		\Delta v'=0 & \text{ in }\Omega \setminus \overline{D'}, \\
		v'=v & \text{ in }\Omega \setminus \overline{D}.
		\end{cases}
		\end{align*}
		Here  $\mathcal{H}^{n-1}(A)$ stands for the $(n-1)$-dimensional Hausdorff measure of a measurable set $A$.
	\end{defi}
	
	In order to introduce the RT, we need some preparation. To begin with let $v\in H^1(\Omega)$ be the unique solution to
	the boundary value problem:
	\begin{equation*}%\label{bvp2}
	\begin{cases}
	\Delta v=0&\text{ in }\Omega,\\
	v=f&\text{ on }\partial\Omega
	\end{cases}
	\end{equation*}
	and consider $w:=u-v$. Then $w$ satisfies
	\begin{equation*}%\label{bvp3}
	\begin{cases}
	\Delta w=0&\text{ in } \Omega\setminus\overline D,\\
	w=-v&\text{ on }\partial D, \\
	w=0 &\text{ on }\partial\Omega.
	\end{cases}
	\end{equation*}
	Denote by $\G_y(\cdot)=\G(\cdot,y)$ be the Green function of $\Delta$ in $\Omega$ with Dirichlet boundary condition on $\partial\Omega$. Let $S[\varphi]$ be the \emph{single-layer potential} defined as $S: H^{-1/2}(\partial G)\rightarrow H^{1}(\Omega\setminus\partial G)$ by 
	\begin{equation}\label{single layer potential}
	S[\varphi](x)=\int_{\partial G}\G(x,y)\varphi(y)\,d\sigma (y),\,\,\varphi\in H^{-1/2}(\partial G),\,x\in\Omega\setminus\partial G
	\end{equation}
	with the surface element $d\sigma (y)$ of $\partial G$.
	Further define the operator 
	\begin{equation}\label{def of R}
	\begin{split}
	&R:H^{-1/2}(\partial G)\rightarrow H^{-1/2}(\partial\Omega) \text{ by } \\
	&R[\varphi]=\partial_\nu S[\varphi]\in H^{-1/2}(\partial \Omega),\,\,\varphi\in H^{-1/2}(\partial G),
	\end{split}	
	\end{equation}
	where $\nu$ is the unit normal of $\partial \Omega$ directed into the exterior of $\Omega$. Of course $R$ is a compact linear operator with the kernel $\partial_{\nu_x}\G(x,y)$.
	
	Now consider the solvability of the boundary integral equation:
	\begin{equation}\label{RT eq}
	R[\varphi](x)=\partial_\nu w(x),\,\,x\in\partial\Omega
	\end{equation}
	with respect to $\varphi\in H^{-1/2}(\partial G)$. Then we have the following theorem.
	
	\medskip\noindent
	\begin{thm}\label{claim1}
		\eqref{RT eq} is solvable if and only if $w$ can be analytically extended to $\Omega\setminus\overline G$ and $w_+\big|_{\partial G}\in H^{1/2}(\partial G)$, where $w_+=w\big|_{\Omega\setminus\overline G}$.
	\end{thm}
	
	\begin{proof}
		We first prove the only if part. Let $\varphi\in H^{-1/2}(\partial G)$ be a solution of \eqref{RT eq}. Let $z=S[\varphi]$. Then by using  \eqref{RT eq}, $z\in H^1(\Omega\setminus\overline G)$ and it satisfies
		\begin{equation*}
		\begin{cases}
		\Delta z=0&\text{ in }\Omega\setminus\overline G,\\
		\partial_\nu z=\p _\nu w &\text{ on } \p \Omega,\\ 
		z=0 &\text{ on }\partial\Omega.
		\end{cases}
		\end{equation*}
		Compare this with
		\begin{equation*}
		\begin{cases}
		\Delta w=0\,\,&\text{in}\,\,\Omega\setminus\overline D,\\
		w=0 &\text{on}\,\,\partial\Omega.
		\end{cases}
		\end{equation*}
		Then by the unique continuation property (UCP) of solutions for the Laplace equation, we have $w=z$ in $\Omega\setminus\overline{(D\cup M)}$, where $M\Subset\Omega$ is the minimal
		domain\footnote{$M$ is the minimal domain in the sense that $u$ can be extended into $\Omega \setminus M$ but not into any $\Omega\setminus N$ for $N\subsetneq M$} such that the harmonic function $z$ can be analytically extended into $\Omega\setminus\overline M$. Hence $w$ can be analytically extended into $\Omega\setminus\overline G$ which implies $w=S[\varphi]$ on $\p G$. Then, since $S: H^{-1/2}(\partial G)\rightarrow H^{1/2}(\partial G)$ is bijective which will be shown in Proposition \ref{claim2}, we have $w_+\big|_{\partial G}\in H^{1/2}(\partial G)$, where we have abused the notation $S$ used to denote the single-layer potential.
		
		Next, we prove the if part. Assume that $w$ can be analytically extended to $\Omega\setminus\overline G$ and $w_+\big|_{\partial G}\in H^{1/2}(\partial G)$. By the bijectivity of  $S: H^{-1/2}(\partial G)\rightarrow H^{1/2}(\partial G)$, there exist a unique $\varphi\in H^{-1/2}(\p G)$ which solves $S[\varphi]=w_+\big|_{\partial G}$ on $\partial G$. Then $z:=S[\varphi]$ satisfies
		\begin{equation*}
		\begin{cases}
		\Delta z=0 &\text{ in }\Omega\setminus\overline G,\\
		z=0 &\text{ on }\partial\Omega,\\
		z=w_+\big|_{\partial G} & \text{ on }\partial G.
		\end{cases}
		\end{equation*}
		Compare this with
		\begin{align*}
		\begin{cases}
		\Delta w=0 & \text{ in }\Omega \setminus \overline{G}, \\
		w=0 & \text{ on }\p \Omega,\\
		w= w_+\big|_{\p G} & \text{ on }\p G.
		\end{cases}
		\end{align*}
		Then, by the uniqueness of solutions to the boundary value problem in $\Omega\setminus\overline G$, we have $z=w$ in $\Omega\setminus G$. Hence
		$$
		R[\varphi]=\partial_\nu S[\varphi]=\partial_\nu z=\partial_\nu w\,\,\text{on}\,\,\partial\Omega.
		$$
	\end{proof}
	
	For any test domain $G$, Theorem \ref{claim1} tells us that the analytic extension of the solution $w\in H^1(\Omega\setminus\overline D)$ to $\Omega\setminus\overline G$ can be tested by the solvability of \eqref{RT eq} with respect to $\varphi\in H^{-1/2}(\partial G)$. We will interpret this solvability using the regularization theory. To begin with recall that $R:H^{-1/2}(\partial G)\rightarrow H^{-1/2}(\partial\Omega)$ is a linear, injective (see Proposition \ref{denseness of Rge(R ast)} of Appendix) and compact operator. Consider the Tikhonov regularized solution $\varphi_\alpha\in H^{-1/2}(\partial G)$ with a regularization parameter $\alpha>0$ of the equation $R[\varphi]=\partial_\nu w$ in $H^{-1/2}(\partial \Omega)$. That is 
	\begin{equation}\label{regularized solution phi alpha}
	\varphi_\alpha=(\alpha I+R^\ast R)^{-1}R^\ast \partial_\nu w,
	\end{equation}
	where $R^\ast:H^{-1/2}(\partial \Omega)\rightarrow H^{-1/2}(\partial G)$ is the adjoint operator of $R$.
	Then by the regularization theory (see for instance Theorem 3.1.10 of \cite{Nakamura}), we have
	\begin{equation}\label{characterization}
	\begin{cases}
	\partial_\nu w\in\text{Rge}(R)&\Longrightarrow\varphi_\alpha\rightarrow\varphi\,\,(\alpha\rightarrow0)
	\,\,\text{in}\,H^{-1/2}(\partial G),\\
	%&\qquad\Longrightarrow\sup_\alpha\Vert\varphi_\alpha\Vert_{H^{-1/2}(\partial G)}<\infty,\\
	\partial_\nu w\not\in\text{Rge}(R)&\Longrightarrow\lim_{\alpha\rightarrow0}\Vert\varphi_\alpha\Vert_{H^{-1/2}(\partial G)}=\infty.\\
	%&\qquad\Longrightarrow\sup_\alpha\Vert\varphi_\alpha\Vert_{H^{-1/2}(\partial G)}=\infty.
	\end{cases}
	\end{equation}
	Hence allowing the limit becomes infinity, we can test whether $\partial_\nu w$ belongs to $\text{Rge}(R)$ by using $\lim_{\alpha\rightarrow0}\Vert\varphi_\alpha\Vert_{H^{-1/2}(\partial G)}$.
	Finally combining this with Theorem \ref{claim1}, we have the following equivalence:
	\begin{equation}\label{equivalence of analytic extension}
	\begin{split}
	&\quad \text{$w$ can be analytically extended to $\Omega\setminus\overline G$ and $w_+\big|_{\partial G}\in H^{1/2}(\partial G)$}\\
	\Longleftrightarrow
	&\quad \text{finite $ \displaystyle\lim_{\alpha\rightarrow 0}\Vert\varphi_\alpha\Vert_{H^{-1/2}(\partial G)}$ exists.}
	\end{split}
	\end{equation}
	Hence we can use $\lim_{\alpha\rightarrow0}\Vert\varphi_\alpha\Vert_{H^{-1/2}(\partial G)}$ as
	a pre-indicator function to test the analytic extension of $w$ across $\partial G$. This is the range test (RT) and its indicator function will be given in Section \ref{rec by RT}.
	
	\section{NRT and its duality between RT}\label{duality}
	\setcounter{equation}{0}
	In this section we will show the duality between the RT and the NRT for the inverse boundary value problem by showing the equivalence of the pre-indicator functions for the both testing methods. There is already the corresponding duality known for the inverse scattering problem by a very short argument (\cite{Nakamura}). We will adopt the argument in \cite{Nakamura} for our inverse problem, but we will add more supplementary arguments. To begin with recall that $R^\ast: H^{-1/2}(\partial\Omega)\rightarrow H^{-1/2}(\partial G)$ was the adjoint operator of $R:H^{-1/2}(\partial G)\rightarrow H^{-1/2}(\partial\Omega)$.
	%and $R^{(\ast)}:L^2(\partial \Omega)\rightarrow L^2(\partial G)$ be the adjoint operator of the operator $R:L^2(\partial G)\rightarrow L^2(\partial\Omega)$ naturally defined by using the kernel  $\partial_{\nu_x}\G(x,y)$. Further denote by $\Lambda_{\partial G}$ and $\Lambda_{\partial\Omega}$ the Bessel potentials on $\partial G$ and $\partial\Omega$, respectively. 
	%Then we have $R^\ast=\Lambda_{\partial G} R^{(\ast)}\Lambda^{-1}_{\partial\Omega}$. Hence $R^\ast$ is also a compact linear operator. 
	Then the pre-indicator function and the indicator function of the NRT for the test domain $G$ are the same and defined by
	\begin{align}\label{pre-indicator function}
	I_{NRT}(G):=\displaystyle\sup_{\zeta\in H^{-1/2}(\partial \Omega),\,\Vert R^\ast \zeta\Vert\le 1}|(\zeta,\partial_\nu w)|,
	\end{align}
	where we have denoted the norm of $H^{-1/2}(\partial G)$ and inner product of $H^{-1/2}(\partial \Omega)$ by $\Vert\cdot\Vert$ and $(\,,\,)$, respectively. Note that the pre-indicator function of RT is based on the operator $R$ while the pre-indicator function of NRT is based on $R^\ast$. This indicates some duality between the RT and NRT.
	
	Now we are ready to state the duality in terms of the equivalence of the pre-indicator functions of RT and NRT as follows.
	\begin{prop} Let $\varphi_\alpha\in H^{-1/2}(\partial G)$ be the Tikhonov regularized solution \eqref{regularized solution phi alpha} of \eqref{RT eq} with regularization parameter $\alpha>0$. Then we have
		\begin{align}\label{equivalence}
		\lim_{\alpha\rightarrow 0}\Vert\varphi_\alpha\Vert^2=\displaystyle\sup_{\zeta\in H^{-1/2}(\partial \Omega),\,\Vert R^\ast \zeta\Vert\le 1}|(\zeta,\partial_\nu w)|.
		\end{align}
		We remark that this equality holds even in the case both sides of this equality are infinite.
	\end{prop}
	\begin{proof}
		By $\varphi_\alpha=(\alpha I+R^\ast R)^{-1}R^\ast \partial_\nu w$ and the denseness of the range $Rge(R^\ast)$ which  will be proved in Proposition \ref{denseness of Rge(R ast)} of Appendix, we first have
		\begin{equation}\label{manipulation1}
		\begin{split}
		&\lim_{\alpha\rightarrow0}\Vert\varphi_\alpha\Vert^2\\
		=&\lim_{\alpha\rightarrow 0}\displaystyle\sup_{\Vert z\Vert\le 1}|(z,\varphi_\alpha)|\\
		=&\lim_{\alpha\rightarrow 0}\displaystyle\sup_{\Vert z\Vert\le 1}|(z, (\alpha I+R^\ast R)^{-1} R^\ast \partial_\nu w)|\\
		=&\lim_{\alpha\rightarrow 0}\displaystyle\sup_{z\in Rge(R^\ast),\,\Vert z\Vert\le 1}|(z, (\alpha I+R^\ast R)^{-1} R^\ast \partial_\nu w)|
		\end{split}
		\end{equation}
		by abusing the notation "$(\,\,,\,)$" to denote the inner product of the Hilbert space $H^{-1/2}(\partial G)$. 
		
		\medskip
		Since $R(\alpha I+R^\ast R)^{-1}=(\alpha I+R R^\ast)^{-1} R$ and \eqref{manipulation1}, we have
		\begin{equation}\label{manipulation2}
		\begin{split}				\lim_{\alpha\rightarrow0}\Vert\varphi_\alpha\Vert^2=&\lim_{\alpha\rightarrow 0}\displaystyle\sup_{z\in Rge(R^\ast),\,\Vert z\Vert\le 1}|((\alpha I+RR^\ast)^{-1}Rz,\partial_\nu w)|\\
		=&\lim_{\alpha\rightarrow 0}\displaystyle\sup_{\Vert R^\ast \zeta\Vert\le 1}|((\alpha I+ RR^\ast)^{-1}RR^\ast \zeta, \partial_\nu w)|,
		\end{split}
		\end{equation}
		where we have put $z=R^\ast\zeta$ with $\zeta\in H^{-1/2}(\partial\Omega)$.
		
		Finally we want to show  
		\begin{equation}\label{manipulation3}
		\displaystyle\lim_{\alpha\rightarrow 0}\displaystyle\sup_{\Vert R^\ast \zeta\Vert\le 1}|((\alpha I+ RR^\ast)^{-1}RR^\ast \zeta, \partial_\nu w)|=
		\displaystyle\sup_{\zeta\in H^{-1/2}(\partial\Omega),\,\Vert R^\ast \zeta\Vert\le 1}|(\zeta,\partial_\nu w)|.
		\end{equation}
		To begin, let $(\mu_n,\varphi_n, \psi_n)$ be the singular system of $R^\ast:H^{-1/2}(\partial\Omega)\rightarrow H^{-1/2}(\partial G)$. The bracket "$(\,\,,\,)$" used in the sequel for proving \eqref{manipulation3} is the inner product of the Hilbert space $H^{-1/2}(\partial\Omega)$. Define $I$ and $J$ by 
		\begin{equation*}%\label{identity for the duality}
		\begin{array}{ll}
		I=\displaystyle\lim_{\alpha\rightarrow0}\displaystyle\sup_{\Vert R^\ast \zeta\Vert\le1}
		\big|\big((\alpha I+R R^\ast)^{-1}R R^\ast \zeta,\,\partial_\nu w\big)\big|,\\
		J=\displaystyle\sup_{\Vert R^\ast \zeta\Vert\le1}\big|( \zeta,\,\partial_\nu w)\big|.
		\end{array}
		\end{equation*}
		We will show the identity $I=J$ even in the case $I,\,J$ are infinite. Observe that 
		\begin{equation}\label{observation for the duality}
		\begin{array}{ll}
		\|R^\ast \zeta\|^2=\displaystyle\sum_{n=1}^\infty\lambda_n|\zeta_n|^2,\\
		\big((\alpha I+R R^\ast)^{-1} R R^\ast \zeta,\,\partial_\nu w\big)=\displaystyle\sum_{n=1}^\infty
		\frac{\lambda_n}{\alpha+\lambda_n}\zeta_n (\partial_\nu w)_n
		\end{array}
		\end{equation}
		with $\lambda_n=\mu_n^2$, $\zeta_n=(\zeta,\,\varphi_n)$, $(\partial_\nu w)_n=
		(\varphi_n,\,\partial_\nu w)$. 
		
		Now we adjust $\zeta$ by multiplying each of its $\zeta_n$ by a unimodular complex number such that $\zeta_n (\partial_\nu w)_n\ge0$. Then the modulus of the second quantity of \eqref{observation for the duality} becomes larger. Further it becomes larger if we assume each $\zeta_n\not=0$. We call such $\zeta\in H^{-1/2}(\partial\Omega)$ positive and denote it by $\zeta^+$. Hence we have
		\begin{equation*}%\label{I with positive sum}
		I=\displaystyle\lim_{\alpha\rightarrow0}\displaystyle\sup_{\Vert R^\ast \zeta^+\Vert\le1}\left|\displaystyle\sum_{n=1}^\infty 
		\frac{\lambda_n}{\alpha+\lambda_n} \zeta_n^+(\partial_\nu w)_n\right|
		\end{equation*}
		with $\zeta_n^+=(\zeta^+,\,\varphi_n),\,n\in{\mathbb N}$.
		Now take a sequence of $\zeta^{+,\ell}\in H^{-1/2}(\partial\Omega)$, $\ell\in{\mathbb N}$ such that
		$\Vert R^\ast \zeta^{+,\ell}\Vert\le 1$, $\ell\in{\mathbb N}$ and
		\begin{equation*}%\label{sequence giving sup}
		\begin{array}{ll}
		I=\displaystyle\sup_{\Vert R^\ast \zeta^+\Vert\le1}\left|\displaystyle\sum_{n=1}^\infty\frac{\lambda_n}{\alpha+\lambda_n}\zeta_n^+ (\partial_\nu w)_n\right|
		=\displaystyle\lim_{\ell\rightarrow\infty}\displaystyle\sum_{n=1}^\infty\left|\frac{\lambda_n}{\alpha+\lambda_n}\zeta_n^{+,\ell}(\partial_\nu w)_n\right|,\\
		J=\displaystyle\lim_{\ell\rightarrow\infty}\left|\displaystyle\sum_{n=1}^\infty \zeta_n^{+,\ell}(\partial_\nu w)_n\right|
		\end{array}
		\end{equation*}
		with $\zeta_n^{+,\ell}=( \zeta^{+,\ell},\,\varphi_n),\,n\in{\mathbb N}$. We can further arrange the sequence $\zeta^{+,\ell},\,\ell\in{\mathbb N}$ such that
		\begin{equation*}%\label{monotonicity for the g-sequence}
		\zeta^{+,\ell'}_n\ge \zeta^{+,\ell}_n,\,n\in{\mathbb N}\,\,\,\text{if}\,\,\,\ell'\ge\ell.
		\end{equation*} 
		Consider a monotone decreasing sequence $\alpha_k,\,k\in{\mathbb N}$ converging to $0$ and define the double sequence $a(k,\ell),\,k,\ell\in{\mathbb N}$ by
		\begin{equation*}%\label{double sequece for the identity}
		a(k,\ell)=\displaystyle\sum_{n=1}^\infty\frac{\lambda_n}{\alpha_k+\lambda_n} \zeta_n^{+,\ell}(\partial_\nu w)_n, \quad k,\ell\in{\mathbb N}.
		\end{equation*}
		Then this double sequence is monotone increasing. That is $a(k',\ell')\ge a(k,\ell)$ if $k'\ge k,\,\ell'\ge\ell$.
		
		Now let $I'=\lim_{k\rightarrow\infty}\lim_{\ell\rightarrow\infty}a(k,\ell)$. Then we have the following implications:
		\begin{equation*}%\label{implications for the identity}
		\begin{array}{ll}
		&  I'<\infty\Longrightarrow \{a(k,\ell)\}\,\,\text{is bounded}\\
		\Longrightarrow & \text{there exists}\displaystyle\lim_{\ell\rightarrow\infty}\displaystyle\lim_{k\rightarrow\infty}a(k,\ell)=J\,\,\text{and}\,\,I'=J<\infty.
		\end{array}
		\end{equation*}
		This implications are reversible. Since the sequence $\alpha_k\rightarrow 0$ was taken arbitrarily, we have shown the identity $I=J$ even in the case the both sides are infinite.
	\end{proof}
	
	\section{Reconstruction of unknown obstacle by RT}\label{rec by RT}
	\setcounter{equation}{0}
	In this section we will show the reconstruction of the unknown obstacle $D$ by the RT. First of all, by \eqref{characterization}, $\lim_{\alpha\to0}\Vert\varphi_\alpha\Vert_{H^{-1/2}(\partial G)}$ can be either finite or infinite.
	Based on this we define the indicator function $I_{RT}(G)$ of RT for a test domain $G$ by
	\begin{equation}\label{test function of RT}
	I_{RT}(G):=
	\begin{cases}
	\displaystyle\lim_{\alpha\to0}\Vert\varphi_\alpha\Vert_{H^{-1/2}(\partial G)}\,\,&\text{if the finite limit exists},\\
	\infty\,\,&\text{if otherwise}.
	\end{cases}		
	\end{equation}
	We call a test domain $G$  positive if it satisfies $I_{RT}(G)<\infty$.	Denote the set of all positive test domains by $\mathcal{P}$.
	
	By using this indicator function $I_{RT}(\cdot)$, we can explain more precisely what is the RT. It is a domain sampling method which uses the indicator function $I_{RT}(\cdot)$ for test domains to reconstruct $D$ or extract some information about the location of $D$. The key to this method is \eqref{equivalence of analytic extension} which gives
	\begin{equation}\label{equivalence of analytic extension by RT indicator function}
	\begin{split}
	&\quad \text{$w$ can be analytically extended to $\Omega\setminus\overline G$ and $w_+\big|_{\partial G}\in H^{1/2}(\partial G)$}\\
	\Longleftrightarrow
	&\quad \text{$I_{RT}(G)<\infty$}.
	\end{split}
	\end{equation}
	
	Then we have the following theorem.
	\begin{thm}\label{RT}
		Let $G$ be a test domain. Then we have the followings.
		\begin{itemize}
			\item[(i)] $\overline D\subset\overline G\Longrightarrow I_{RT}(G)<\infty$.
			\item[(ii)] Let $\overline D\not\subset\overline G$ and recall $u\in H^1(\Omega\setminus\overline D)$ is the solution to the boundary value problem \eqref{Main equation}. If $u$ admits an analytic extension across $\partial D$ from $\Omega\setminus\overline D$ into the whole $\overline D\setminus G$, then $I_{RT}(G)<\infty$.
			
			\item[(iii)] Let $\overline D\not\subset\overline G$. If the above $u$ does not admit an analytic extension into $\Omega\setminus\overline G$, then $I_{RT}(G)=\infty$.
		\end{itemize}
		As a consequence, if $u$ does not admit an analytic extension across $\partial D$, then we have the following reconstruction of $\overline D$ as $\overline D=\cap_{G\in\mathcal{P}}\overline G$.
	\end{thm}
	\begin{proof}
		If $\overline D\subset\overline G$, then $w$ is analytic in $\Omega\setminus\overline G$. Also since $\partial G\subset\Omega\setminus D$ and $w\in H^1(\Omega\setminus \overline D)$, $w_+\big|_{\partial G}\in H^{1/2}(\partial G)$. Hence $I_{RT}(G)<\infty$. This proves (i). 
		
		Next let $\overline D\not\subset\overline G$ and assume that $u$ admits an analytic extension across $\partial D$ from $\Omega\setminus\overline D$ into the whole $\overline D\setminus G$. Then $w$ is analytic in $\Omega\setminus G$. Hence $I_{RT}(G)<\infty$ which proves (ii). 
		
		Now assume $\overline D\not\subset\overline G$ and $u$ does not admit an analytic extension across $\partial D$. Then we prove $I_{RT}(G)=\infty$ by contradiction.
		Suppose $I_{RT}(G)<\infty$. Then $w$ is analytic in $\Omega\setminus\overline G$. By the assumption on $u$, this implies $\partial D\cap(\Omega\setminus\overline G)=\emptyset$. That is $\partial D\subset\overline G$. But since $D,\,G$ are connected open sets compactly embedded in $\Omega$, we have $\overline D\subset\overline G$. This contradicts to $\overline D\not\subset\overline G$ . Hence we must have $I_{RT}(G)<\infty$ which proves (iii). 
		
		Finally we will prove the last statement. From (i) and (iii), we have $\overline D\subset\overline G\Longrightarrow G\in\mathcal{P}$ and $\overline D\not\subset\overline G\Longrightarrow G\not\in\mathcal{P}$, respectively. Hence we have
		$\overline D\subset\overline G\Longleftrightarrow G\in\mathcal{P}$ which immediately gives $\overline D=\cap_{G\in\mathcal{P}}\overline G$.
	\end{proof}
	
	\begin{rem}\label{remark RT}
		We could have defined the indicator function $I_{RT}(G)$ for a test domain $G$
		as
		\begin{equation}
		I_{RT}(G)=
		\begin{cases}
		\displaystyle\lim_{\alpha\to0}\sup_{\beta<\alpha}\Vert\varphi_\alpha-\varphi_\beta\Vert_{H^{-1/2}(\partial G)}\,\,&\text{if the finite limit exists},\\
		\infty\,\,&\text{if otherwise}.
		\end{cases}		
		\end{equation}
		Then we have
		\begin{equation}\label{equivalence of analytic extension2}
		\begin{split}
		&\quad \text{$w$ can be analytically extended to $\Omega\setminus\overline G$ and $w_+\big|_{\partial G}\in H^{1/2}(\partial G)$}\\
		\Longleftrightarrow   &\quad I_{RT}(G)=0.
		\end{split}
		\end{equation}
		Hence we can have Theorem \ref{RT} in terms of this indicator function with obvious changes. We note that it is nice to have \eqref{equivalence of analytic extension2} instead of $I_{RT}(G)<\infty$ in \eqref{equivalence of analytic extension by RT indicator function}. However the pre-indicator $\lim_{\alpha\to0}\sup_{\beta<\alpha}\Vert\varphi_\alpha-\varphi_\beta\Vert_{H^{-1/2}(\partial G)}$ of this new $I_{RT}(\cdot)$ loose the direct connection with the pre-indicator function $I_{NRT}(\cdot)$ of the NRT which will be defined in the next section.
	\end{rem}

	\section{Reconstruction of unknown obstacle by NRT}\label{rec by NRT}
	
	In this section we will show the reconstruction of the unknown obstacle $D$ by the NRT. Needless to say that we can have the reconstruction for the NRT by using the duality. But we pursue a way to provide the reconstruction for the NRT without using the duality.
	We start with some preparation necessary for showing the reconstruction. To avoid heavy notations, let $X=H^{-1/2}(\p G)$, $Y=H^{-1/2}(\p \Omega)$. Recall that the operator $R:X\to Y$ was defined by $R[\varphi]=\p_{\nu}S[\varphi]$ on $\p \Omega$. Hence 
	\[
	R[\varphi](x)=\int_{\p G}\p _{\nu_x}\G(x,y)\varphi(y)\, d\sigma(y),\,\, x\in \p \Omega,
	\]
	for $\varphi\in X$.
	
	We also let $X^\ast =H^{1/2}(\p G)$ and $Y^\ast=H^{1/2}(\p \Omega)$ be the dual spaces of $X$ and $Y$, respectively. Also let $(\cdot, \cdot)_X$ and $(\cdot, \cdot)_Y$ be the inner products in $X$ and $Y$, respectively. Then the adjoint operator $R^\ast:Y\to X$ of $R$ can be given by the relation 
	\[
	(\phi, R^\ast \psi)_{X}=(R\phi, \psi)_Y, \,\, \phi\in X,\ \psi \in Y.
	\]
	Further let $J_X$ and $J_Y$ be the isometric isomorphism $J_X: X\to X^\ast $, $J_Y:Y\to Y^\ast $ defined by 
	\[
	(J_X\phi)(\psi)=(\psi,\phi)_X,\,\,\phi,\psi\in X,
	\]
	and define $J_Y$ in a similar way. Then the dual operator $R^{(\ast)}:Y^\ast \to X^\ast$ of the operator $R$ and the adjoint operator $R^\ast$ have the relation $R^\ast =J_X^{-1}R^{(\ast)}J_Y$. Moreover a direct computation yields that 
	\begin{align}\label{R(ast)}
	(R^{(\ast)}\eta)(y)=\int_{\p \Omega}\p_{\nu_x}\G(x,y)\eta (x)\, d\sigma(x), \,\,y\in \p G
	\end{align}
	for any $\eta\in Y^\ast $. 
	
	Observe that 
	\begin{align}\label{hat zeta}
	\norm{R^\ast \zeta}_{X}=\norm{J_X^{-1}R^{(\ast)}J_Y\zeta}_{X} =\norm{R^{(\ast)}J_Y\zeta}_{X^\ast} =\norm{R^{(\ast)}\widehat{\zeta}}_{X^\ast}
	\end{align}
	where $\widehat{\zeta}=J_Y \zeta \in Y^\ast$.  Let $W[\varphi]$ be the \emph{double-layer potential}  for $\varphi\in Y^\ast$ by 
	\begin{align}\label{double-layer potential}
	W[\varphi](x):= \int_{\p \Omega}\p_{\nu _y} \G_0 (x,y)\varphi(y)\, d\sigma(y),
	\end{align}
	where $\G_0(x,y)=\frac{1}{4\pi |x-y|}$ for $x\not=y$ and $\nu_y$ is the unit outer normal of $\p \Omega$. Then we have the following representations of $I_{NRT}(G)$. 
	
	\begin{prop}\label{Prop: Equivalent of pre-indicator}
		The indicator function defined by  \eqref{pre-indicator function} has the following representation  
		\begin{align}\label{pre-indicator function_equivalence}
		I_{NRT}(G)=\sup_{\widehat{\zeta}\in Y^\ast,\  \norm{R^{(\ast)}\widehat{\zeta}}_{X^\ast}\leq 1}\left|\int_{\p \Omega}\widehat{\zeta}\p_\nu w\, d\sigma(x)\right|. 
		\end{align}
		%%%%%%%%%%%%%%%%%%%
		%\displaystyle\sup_{\zeta\in Y,\,\Vert R^\ast \zeta\Vert_{X}\le 1}|(\zeta,\partial_\nu w)|
		%%%%%%%%%%%%%%%%%
		Moreover, \eqref{pre-indicator function_equivalence} can be rewritten as
		\begin{equation}\label{indicator function NRT}
		\begin{cases}
		I_{NRT}(G)=\\
		\displaystyle\sup_{\varphi \in Y^\ast , \norm{R^{(\ast)} W[\varphi]}_{X^\ast }\leq 1} \left|\int_{\p D}\left(W[\varphi](x)\p_\nu w(x)+\p_\nu W[\varphi](x)v(x)\right) d\sigma(x)\right|,
		\end{cases}
		\end{equation}
		where $w=u-v$ and $v\in H^1(\Omega)$ is the solution to $\Delta v=0$ in $\Omega$ with $v=f$ on $\p \Omega$. 
	\end{prop}
	
	\begin{proof}
		By the definitions of $\widehat\zeta$ given after \eqref{hat zeta} and $J_Y$, we have 
		\begin{align*}
		(\zeta, \p _\nu w)_{Y} =\overline{(J_Y\zeta)(\p _\nu w)}=\overline{\widehat{\zeta}(\p _\nu w)}=\overline{\int_{\p \Omega}\widehat{\zeta}\p_\nu w\, d\sigma(x)}. 
		\end{align*}
		Then combining this with \eqref{hat zeta} we have \eqref{pre-indicator function_equivalence}.

		Next, by the jump formula of the double-layer potential $W[\varphi]$, we have 
		\begin{align*}
		W[\varphi](x)=	-\frac{1}{2}\varphi(x)+\int_{\p \Omega}\p_{\nu _y} \G_0 (x,y)\varphi(y)\, d\sigma(y),\,\,x\in \p \Omega.
		\end{align*}
		Note that $W[\varphi]$ is harmonic in $\Omega$ and the Dirichlet boundary value problem for the Laplace equation in $\Omega$ is unique. Then, by the compactness of 
		\begin{align*}
		Y^\ast \ni \varphi \mapsto \int_{\p \Omega} \p_{\nu _y} \G_0 (x,y)\varphi(y)\, d\sigma(y) \in Y^\ast 
		\end{align*}
		and the Fredholm alternative, there is a unique solution $Q[\widehat{\zeta}]:=\varphi\in Y^\ast $ satisfying the boundary integral equation 
		\[
		W[\varphi](x)=\widehat{\zeta}(x),\,\,x\in \p \Omega,
		\]
		and $Q[\widehat\zeta]$ depends continuously on $\widehat \zeta\in Y^\ast$. By $w=0$ on $\p \Omega$ and $w=-v$ on $\partial D$, we have 
		\begin{align}
		\int_{\p \Omega}W[\varphi](x)\p _\nu w(x)\, d\sigma(x)=\int_{\p D}\left(W[\varphi](x)\p_\nu w(x)+\p _\nu W[\varphi](x)v(x) \right)d\sigma(x)
		\end{align}
		by the Green formula, where $\nu $ is the unit outer normal on $\p D$. This immediately implies \eqref{indicator function NRT} and completes the proof.
	\end{proof}

	Now we are ready to start considering the reconstruction of the unknown obstacle $D$ by the NRT. By Proposition \ref{Prop: Equivalent of pre-indicator}, it suffices to consider the indicator function \eqref{indicator function NRT} for studying the reconstruction by the NRT. The reconstruction by the NRT is a sampling method to test the test domains using the indicator function defined by \eqref{indicator function NRT}. Likewise before for the reconstruction by the RT, a test domain $G$ is called \emph{positive} or \emph{no-response} if $I_{NRT}(G)$ is finite.
	
	\begin{rmk}${}$
		The meaning of no-response is as follows. If we mask the obstacle $D$ by a test domain $G$ i.e. $\overline D\subset\overline G$, then we don't have any huge response i.e. $I(G)<\infty$ (see the proof of Theorem \ref{Thm: NRT} given later).
	\end{rmk}
	
	\medskip
	By replacing $I_{RT}(\cdot)$ in Theorem \ref{RT} by $I_{NRT}(\cdot)$, we have the following reconstruction of an unknown obstacle $D$ by the NRT. 
	
	\begin{thm}\label{Thm: NRT}
		For a test domain $G$, we have the followings.
		\begin{itemize}
			\item[(i)] $\overline D\subset\overline G\Longrightarrow I_{NRT}(G)<\infty$.
			
			\item[(ii)] Let $\overline D\not\subset\overline G$ and the boundary $\partial(D\setminus\overline G)$ of $D\setminus\overline G$ be Lipschitz continuous. If the solution $u\in H^1(\Omega\setminus\overline D)$ to the forward problem admits an analytic extension across $\partial D$ from $\Omega\setminus\overline D$ into the whole $\overline D\setminus G$, then $I_{NRT}(G)<\infty$.
			
			\item[(iii)] As an additional assumption, we assume that $\Omega$ and $G$ are \emph{$C^2$-spherical domains}, i.e. $\Omega$ and $G$ are diffeomorphic to an open ball up to their boundaries. Let $\overline D\not\subset\overline G$. If the above $u$ does not admit an analytic extension into $\Omega\setminus\overline G$, then $I_{NRT}(G)=\infty$.
		\end{itemize}
		As a consequence, if $\Omega$ is a $C^2$-spherical domain and also the above solution $u$ does not admit an analytic extension across $\partial D$, then we have the following reconstruction of $\overline D$ as the intersection of all $\overline G$ for positive $C^2$-spherical test domain $G$.
	\end{thm}
	
	Before getting into the proof, we would like to emphasize that the proof of this theorem is much more involved than that of Theorem \ref{RT}. Especially proving (iii). The key to prove (iii) is the analytic extension of solutions for the Laplace equation. This can be analyzed by using the following well known lemma given in \cite[Lemma 3.2]{Potthast}. The proof of (iii) is almost the same as the proof of Theorem \ref{RT} in that paper, but we need to adopt the proof there to our situation.
	
	\begin{lem}
		Let $\mathcal{O}\Subset\Omega$ be a domain with $C^2$ boundary, and $\mathcal{O}_e:=\Omega \setminus \overline{\mathcal{O}}$. Let $u$ be analytic in $\mathcal{O}_e$, and consider the following set associated with the Taylor coefficients of $u$ at $z\in \mathcal{O}_e$ with $\dist(z,\p \Omega)>\rho$ for some $\rho>0$:
		\begin{align}\label{set of Taylor coefficients}
		\left\{a_{\ell}(z):=\sup_{h\in{\mathbb R}^3,\,|h|=1}\rho^\ell \frac{{|(h\cdot \nabla )^\ell u(z)|}}{\ell !}:\ \ell \in {\mathbb Z}_+:=\N\cup\{0\} \right\}.
		\end{align}
		If this is uniformly bounded in a neighborhood of all such $z$, then $u$ can be extended into an open neighborhood of $\overline{\mathcal{O}_e}$. In other words, there exists a set $\mathcal{O}'$ with $\overline{\mathcal{O}'}\subset \mathcal{O}$ such that $u$ is extensible into $\mathcal{O}'_e:=\Omega\setminus \overline{\mathcal{O}'}$.
	\end{lem}

	\begin{proof}[Proof of Theorem \ref{Thm: NRT}]
		We first consider the case $\overline D\subset \overline G$. Then $w$ has an analytic extension to $\Omega\setminus G$ which also implies $w_+\big|_{\partial G}\in H^{1/2}(\partial G)$. Hence by Theorem \ref{claim1}, $R[\varphi]=\partial_\nu w$ has a solution $\varphi\in H^{-1/2}(\partial G)$. Then we have
		$$
		\begin{array}{ll}I_{NRT}(G)&=\displaystyle\sup_{\zeta\in H^{-1/2}(\partial\Omega),\,\Vert R^\ast\zeta\Vert_{H^{-1/2}(\partial G)}\leq 1}\left|(\zeta, R[\varphi])_{H^{-1/2}(\partial\Omega)}\right|\\
		&=\displaystyle\sup_{\zeta\in H^{-1/2}(\partial\Omega),\,\Vert R^\ast\zeta\Vert_{H^{-1/2}(\partial G)}\leq 1}\left|(R^\ast\zeta, \varphi)_{H^{-1/2}(\partial G)}\right|<\infty.
		\end{array}
		$$
		This proves (i). 
		
		Next we prove (ii). By the assumption, $w$ is analytic outside $G$ which also implies that $w_+\big|_{\partial(D\setminus \overline G))}\in H^{1/2}(\partial(D\setminus\overline G))$, $R[\varphi]=\partial_\nu w$ has a solution $\varphi\in H^{-1/2}(\partial(D\setminus\overline G))$, where $R:H^{-1/2}(\partial(D\setminus\overline G))\rightarrow H^{-1/2}(\partial\Omega)$. By \eqref{pre-indicator function_equivalence} we have
		$$
		\begin{array}{ll}
		I_{NRT}(G)&=\displaystyle\sup_{\widehat\zeta\in H^{1/2}(\partial\Omega),\,\Vert R^{(\ast)}\widehat\zeta\Vert_{H^{1/2}(\partial G)}\le1}\left|\int_{\partial\Omega} \widehat\zeta(x)\partial_\nu w(x)\,d\sigma(x)\right|\\
		&=\displaystyle\sup_{\widehat\zeta\in H^{1/2}(\partial\Omega),\,\Vert R^{(\ast)}\widehat\zeta\Vert_{H^{1/2}(\partial G)}\le1}\left|\int_{\partial(D\setminus\overline G))} R^{(\ast)}[\widehat\zeta](x)\varphi(x)\,d\sigma(x)\right|.
		\end{array}
		$$
		Here note that by the definition \eqref{R(ast)} of $R^{(\ast)}$ and the mapping property of the associated double-layer potential naturally defined from $R^{(\ast)}$ which will be denoted by the same notation, we have
		$\Vert R^{(\ast)}\widehat\zeta\Vert_{H^{1/2}(\partial G)}\le1$ implies  $\Vert R^{(\ast)}\widehat\zeta\Vert_{H^1(\Omega\setminus\overline G)}\le C$ for some general constant $C>0$. Then by $D\setminus\overline G\subset\Omega\setminus\overline G$ and the continuity of trace to
		$\partial(D\setminus\overline G)$, we have
		$\Vert R^{(\ast)}\widehat\zeta\Vert_{H^{1/2}(\partial(D\setminus\overline G))}\le C$ for some general constant $C>0$. Hence we have $I_{NRT}(G)<\infty$. This proves (ii).
		
		Our next task is to prove (iii). Assume that $u$ does not admit an analytic extension into $\Omega\setminus\overline G$. Let $\{G_t\}_{t\in[0,1]}$ be a homotopy, hence $G_t$ depends continuously on $t\in[0,1]$ and $G_0=\Omega,\,G_1=G$, and moreover it satisfies $G_{t'}\Subset G_t$ for $t'>t$, $t,\,t'\in[0,1]$ and each $G_t$ with $t\in(0,1)$ has the same regularity and topological property as $G$. As for the existence of such a homotopy see Theorem 6.3 of \cite{HPNS}. Since $u$ cannot be analytically extended into the set $\Omega \setminus \overline{G}$, then there must exist a maximal parameter $t_0\in (0,1)$ such that $u$ can be analytically extended into $\Omega\setminus G_t$ for any $t<t_0$ but not into any $\Omega \setminus G_t$ for all $t>t_0$. Hence, for a given $\rho>0$, there exist $z_0\in \Omega \setminus G_{t_0}$ with $\dist(z_0,\p \Omega):=\inf_{y\in\partial\Omega}|y-z_0|>\rho$  such that the set \eqref{set of Taylor coefficients} is not uniformly bounded in any neighborhood $V(z_0)$ of $z_0$. Put
		\begin{align}\label{distance of z0 to G}
		\rho_0 :=\text{dist}(z_0,\overline G).
		\end{align}
		Take $V(z_0)$ as a subset of a ball $B_{\rho_0/2}(z_0)$ with radius $\rho_0/2$ centered at $z_0$. Fix any $h \in \R^3$ with $|h|=1$ and define
		\begin{align}\label{beta}
		\beta(z,\ell):=\Vert (h\cdot\nabla)^\ell\mathcal{G}(\cdot,z)\Vert_{H^1(G\cup (\Omega\setminus\overline{\Omega_\epsilon}))}, 
		\end{align}
		for $\ell \in \Z_+$, $z\in V(z_0)$, where $\Omega_\epsilon:=\{x\in\Omega:\,\text{dist}(x,\,\partial\Omega)>\epsilon\}$ with a small enough $\epsilon>0$. Then the $H^1(G)$-norm of the function 
		\begin{align}\label{the testing function H}
		\HH(\cdot,z):= \frac{1}{2\sigma\kappa\beta (z,\ell)}(h \cdot \nabla )^\ell \G (\cdot , z)
		\end{align}
		is bounded by $(2\sigma\kappa)^{-1}$ for $z$ in $V(z_0)$, where $\sigma=\Vert R^{(\ast)}\Vert$ and $\kappa>0$ is the norm of the trace operator $H^1(G)\rightarrow H^{1/2}(\partial G)$.
		
		Now for any fixed $z\in V(z_0)\setminus G_{t_0}$, there exist domains $M(z)\Subset \widetilde M(z)$ with the following properties: 
		\begin{itemize}
			\item[(1)] the boundaries $\partial M(z)$, $\partial \widetilde M(z)$ of $M(z)$, $\widetilde M(z)$ are $C^2$,
			
			\item[(2)] $u$ is analytic in $\Omega\setminus M(z)$,
			
			\item[(3)] $z\notin \widetilde M(z)$,
			
			\item[(4)] $\Omega\setminus\overline {M(z)}$, $\Omega\setminus\overline{\widetilde M(z)}$ are connected,
			
			\item[(5)] $G\subset M(z)$.
			
		\end{itemize}
		%Then one obtains that the Cauchy data $\left\{u|_{\p M[z]} ,  \left.\frac{\p u}{\p \nu }\right|_{\p M[z]} \right\}$ of the extension of $u$ are well-defined.

		Since $W$ has a dense range (see Appendix), there exists a sequence $\varphi_n^z\in H^{1/2}(\partial\Omega),\,n\in\N$ such that 
		\begin{align}\label{Runge1}
		\norm{W[\varphi_n^z]-\HH(\cdot,z)}_{H^{1/2}(\partial\widetilde M(z))}\rightarrow0,\,\,n\rightarrow\infty.
		\end{align}
		Here note that both $W[\varphi_n^z],\,\HH(\cdot,z)$ satisfy the Laplace equation in $\widetilde M(z)$. Then by the interior regularity estimate for solutions of the Laplace equation, we have
		\begin{equation}\label{approximation of HH}
		\Vert W[\varphi_n^z]-\HH(\cdot,z)\Vert_{H^1(M(z))}\rightarrow0,\,\,n\rightarrow\infty
		\end{equation}
		and hence
		\begin{align}\label{Runge2}
		\norm{W[\varphi_n^z]-\HH(\cdot,z)}_{H^1({M(z)})}\leq (2\sigma\kappa)^{-1}
		\end{align}
		for large enough $n$.
		From the above property (5),
		\begin{align}
		\Vert W[\varphi_n^z]\Vert_{H^1(G)} \leq\Vert W[\varphi_n^z]-\HH(\cdot,z)\Vert_{H^1(G)}+ \Vert\HH(\cdot,z)\Vert_{H^1(G)}
		\leq(\sigma\kappa)^{-1}
		\end{align}
		for any fixed $z\in V(z_0)$ and a large enough $n$. This implies 
		$$\label{test bound condition}
		\Vert R^{(\ast)}W[\varphi_n^z]\Vert_{H^{1/2}(\partial G)}\le 1
		$$
		for any fixed $z\in V(z_0)$ and a large enough $n$. 
		
		Now recall \eqref{indicator function NRT}. Based on this for any $n\in \N$, consider 
		\begin{align}\label{pre-indicator function in NRT}
		I_{	NRT}^{(n)}(G):=\left|\int_{\p D}\left(W[\varphi_n^z](x)\p_\nu w(x)+\p_\nu W[\varphi_n^z](x)v(x)\right) d\sigma(x)\right|.
		\end{align}
		Then by the Green formula, we have 
		\begin{align}\label{no extension will blowup}
		\begin{split}
		I_{	NRT}^{(n)}(G)=& \left|\int_{\p D}\left(W[\varphi_n^z](x)\p_\nu w(x)-\p_\nu W[\varphi_n^z](x)w(x)\right) d\sigma(x)\right|\\
		=&  \left|\int_{\gamma}\left(W[\varphi_n^z](x)\p_\nu w(x)-\p_\nu W[\varphi_n^z](x)w(x)\right) d\sigma(x)\right|,
		\end{split}
		\end{align}
		where $\gamma=\p (M(z)\cap D)$ and the normal on $\gamma$ is pointing into the exterior of the finite region closed by the curve $\gamma$. Since $\gamma \subset \overline{M(z)}$, by taking the limit $n\to \infty$ to \eqref{no extension will blowup} and using \eqref{the testing function H}, we have \begin{align}\label{almost final estimate}
		\begin{split}
		&\lim_{n \to \infty}I^{(n)}_{NRT}(G)\\
		=& \left|\int_{\gamma}\left(\frac{\p w}{\p \nu}(x)\HH(x,z)-\frac{\p \HH(x,z)}{\p \nu(x)}w(x) \right)d\sigma(x)\right| \\
		=&\frac{1}{\beta(z,\ell)}\left|\int_{\gamma}\left(\frac{\p w}{\p \nu}(x)(h\cdot \nabla )^\ell \G (x , z) - \frac{\p }{\p \nu(x)}\left((h \cdot \nabla )^\ell \G (x , z)\right)w(x) \right)d\sigma(x)\right| \\
		=& \left|\frac{1}{\beta(z,\ell)}(h\cdot \nabla )^\ell w(z)- \frac{1}{\beta(z,\ell)}(h\cdot \nabla )^\ell w_{ext}(z)\right|,
		\end{split}
		\end{align}
		where 
		\begin{align}\label{uext}
		w_{ext}(z):=\int_{\p \Omega}\left(\frac{\p w}{\p \nu}(y)\G(y,z)-\frac{\p \G(y,z)}{\p \nu(y)}w(y) \right)d\sigma(y).
		\end{align}
		For any fixed $\rho>0$ specified later, the set \eqref{set of Taylor coefficients} is not uniformly bounded in any neighborhood $V(z_0)$ of $z_0$. Hence, for any $k\in\N$, there exist $\ell_k \in \N$ and $z_k\in V(z_0)\setminus G_{t_0}$ such that 
		\begin{align}\label{estimate of w}
		\left|\frac{\rho^{\ell_k}}{\ell_k !}(h \cdot \nabla)^{\ell_k} w(z_k) \right|\geq k.
		\end{align}
		On the other hand, by the analyticity of $\mathcal{G}(y,z)$ with respect to $z\in V(z_0)\setminus G_{t_0}$ for $y\in G\cup(\Omega\setminus\overline{\Omega_\epsilon})$, there exist constants $c>0$ and $\rho>0$ such that
		\begin{align}\label{another estimate of beta}
		|\beta(z,\ell)|\leq c\frac{\ell !}{\rho^\ell}, \quad z\in V(z_0)\setminus G_{t_0}, \quad \ell \in \N.
		\end{align}
		Further from the definition of $\beta(z,\ell)$, we have
		\begin{equation}\label{estimate of wk}
		\left|\frac1{\beta(z,\ell)}
		(h\cdot\nabla)^\ell w_{\text{ext}}(z)\right|\le c',\,\,z\in V(z_0)\setminus G_{t_0},\,\,\ell\in\N
		\end{equation}
		for some constant $c'>0$.
		Hence combining \eqref{almost final estimate}--\eqref{estimate of wk}, we have
		
		\begin{align*}
		\begin{split}
		I_{NRT}^{(k)}(G)=&\left|\int_{\p D}\left(W[\varphi_k](x)\p_\nu w(x)-\p_\nu W[\varphi_k](x)w(x)\right) d\sigma(x)\right|\\
		\geq & \frac{1}{c}\left|\frac{\rho^{\ell_k}}{\ell_k !} (h \cdot \nabla )^{\ell_k}w(z_k) \right| -\wt c\,\,\,\,\text{with some constant $\tilde c>0$} \\
		\to & \infty \quad \text{ as } k\to \infty.
		\end{split}
		\end{align*}
		Thus we have
		\begin{align*}
		I_{NRT}(G)\geq \lim_{k\to \infty}I_{NRT}^{(k)}(G)=\infty.
		\end{align*}
		This completes the proof of (iii). The rest of the proof is the same as that for Theorem \ref{RT}.
	\end{proof}
	
	\begin{rem}\label{remark NRT}
		We could have defined the indicator function $I_{NRT}(G)$ for a test domain $G$ as
		\begin{equation}
		I_{NRT}(G):=\displaystyle\lim_{\epsilon\to0}\displaystyle\sup_{\zeta\in H^{-1/2}(\partial \Omega),\,\Vert R^\ast \zeta\Vert\le \epsilon}|(\zeta,\partial_\nu w)|.
		\end{equation}
		Then we can say the same as before in Remark \ref{remark RT}.
	\end{rem}
	
	\section{Convex polygonal $D$}\label{sec convex case}
	In this section we will provide a proof of Theorem \ref{convex}. To begin with we give the definition that a corner of a polygon has an irrational angle as follows.
	
	\begin{defi}
		Let $0<\theta<2\pi$ be the angle of a corner of a polygon $D$. This angle is called irrational if $\theta=2\pi\alpha$ with $\alpha\not\in{\mathbb Q}$.
	\end{defi}
	
	\medskip
	Concerning the analytic extension of $u$ across $\partial D$ we have the following proposition. 
	
	\begin{prop}\label{no analytic extension}
		Let $D$ be a convex polygon. Assume that $D$ satisfy one of the following conditions.
		\begin{itemize}
			\item[{\rm (i)}] All the corners of $D$ have irrational angles.
			\item[{\rm (ii)}] $D$ satisfies the distance property \eqref{distance property}.
		\end{itemize}
		%and $u\in H^1(\Omega\setminus\overline D)$ be the solution to \eqref{Main equation}
		%with non-identically zero $f\in H^{1/2}(\partial\Omega)$. 
		Then $u$ cannot analytically extend across this corner.
	\end{prop}
	\begin{proof} This can be easily proved along the same way following the proofs for Lemma 3.1 and Lemma 3.2 in \cite{Friedman}. 
		More precisely consider the function $u^e(x)$ in the proof of Lemma 3.1 of \cite{Friedman} which satisfies the Dirichlet boundary condition on $P_1\cap P_2\cap B_\epsilon(a_0)$ and assume that it can be analytically continued across $a_0$ into $D$. This $u^e$ corresponds to our $u$. Then $u^e(x)$
		becomes
		$$u^e(x)=\displaystyle\sum_{k=1}^\infty b_k^\epsilon r^k\sin(k\varphi)
		$$
		with $b_k^\epsilon\sin(k\alpha)=0$, $k\in{\mathbb N}$. Here we have 
		\begin{equation}\label{indices}
		\begin{cases}
		\text{\rm case 1:}\,\,\sin(k\alpha)\not=0,\,k\in{\mathbb N}\,\,\,\text{if $\alpha/\pi\not\in{\mathbb Q}$}\Longrightarrow b^\epsilon_k=0,\,\,k\in{\mathbb N},\\
		\\
		\text{\rm case 2:}\,\,\sin(k\alpha)=0,\,\,k\in{\mathbb N},\,p|k\,\,\,\text{if $\alpha/\pi=q/p\in{\mathbb Q}$}\\
		\qquad\qquad\qquad\qquad\qquad\qquad\qquad \quad \text{\rm with $p,q\in{\mathbb N},\, (p,q)=1$},
		\end{cases}
		\end{equation}
		where $p|k$ and $(p,q)=1$ mean that $p$ is a divisor of $k$ and
		$p,q$ do not have any common divisor, respectively. The case (i) in the theorem corresponds to the case 1 and the case (ii) in the theorem we have both the cases 1 and 2. Hence in the case (i), we have $f=u\big|_{\partial\Omega}\equiv0$ by the UCP which contradicts to $f\not\equiv0$. For the case 2, note that we have $k=mp,\,m\in{\mathbb N}$. Then arguing as in \cite{Friedman} using the distance condition, $u$ can be analytically continued inside the whole $D$. Hence $u$ satisfies the homogeneous Dirichlet boundary value problem in $D$ for the Laplace equation, which yields $u=0$ in $D$. Then the UCP implies $f\equiv0$ which contradicts to $f\not\equiv0$.
	\end{proof}
	
	We will first show Theorem \ref{convex} by using the RT. \begin{thm}\label{RT convex}
		Let $D$ be the same as in Proposition \ref{no analytic extension}. Let $G$ be a convex polygonal test domain. Then we have the followings.
		\begin{itemize}
			\item[(i)] $\overline D\subset\overline G\Longrightarrow I_{RT}(G)<\infty$.
			\item[(ii)] $\overline D\not\subset\overline G\Longrightarrow I_{RT}(G)=\infty$. 
		\end{itemize}
		As a consequence we have the following reconstruction of $\overline D$ as $\overline D=\cap_{G\in\mathcal{P}}\overline G$.
	\end{thm}
	\begin{proof}
		(i) can be proved in the same way as that of Theorem \ref{RT}. If $\overline D\not\subset \overline G$, there is at least one vertex of $D$ lying outside of $\overline G$. Then by using Proposition \ref{no analytic extension}, we easily have $I_{RT}(G)=\infty$.
	\end{proof}
	
	Finally we show Theorem \ref{convex} by using NRT.
	\begin{thm}
		Assume the existence of a priori convex polygon $D_0$. Let $D$ be the same as in Proposition \ref{no analytic extension} and let $G\subset D_0$ be a convex polygonal test domain. Then we have the followings.
		\begin{itemize}
			\item[(i)] $\overline D\subset\overline G\Longrightarrow I_{NRT}(G)<\infty$.
			\item[(ii)] $\overline D\not\subset\overline G\Longrightarrow I_{NRT}(G)=\infty$. 
		\end{itemize}
		As a consequence we have the following reconstruction of $\overline D$ as $\overline D=\cap_{G\in\mathcal{P}}\overline G$, where $\mathcal{P}$ denotes the set of all positive convex polygonal test domain $G\subset D_0$.
	\end{thm}
	\begin{proof}
		The proof of (i) is the same as that of (i) in Theorem \ref{Thm: NRT}. Also by considering the homotopy $\{G_t\}_{t\in[0,1]}$ such that $G_0=D_0$, $G_1=G$ and each $G_t$ is a convex polygon, we can just repeat the proof of (iii) in Theorem \ref{Thm: NRT} to prove (ii).
	\end{proof}

	\appendix

	\section*{Appendix}\label{Appendix}
	\setcounter{equation}{0}
	\renewcommand{\theequation}{A.\arabic{equation}}
	\renewcommand{\theprop}{A.\arabic{prop}}
	\setcounter{prop}{0}
	In this appendix we will give the proofs of the bijectivity of $S$, the denseness of the ranges of operators $R^\ast$ and $W$. We give unified proofs which are valid for both the two dimensional case and the three dimensional case. We first prove the bijectivity of $S$.
	\begin{prop}\label{claim2} 
		The operator 
		$$
		S: H^{-1/2}(\partial G)\rightarrow H^{1/2}(\partial G)\,\,\text{is bijective}.
		$$
	\end{prop}

	\begin{proof}
		Decompose $S$ into
		$$
		S=S_0+S_1,
		$$
		where  $S_0$ is the operator defined likewise the operator $S$ by the potential
		\begin{equation}
		\begin{cases}
		\frac1{4\pi|x-y|}\,\,&\text{if $n=3$},\\
		-\frac{1}{2\pi}\log|x-y|\,\,&\text{if $n=2$},
		\end{cases}
		\end{equation}
		for $x\not=y$. Then it is well known that
		$S_0: H^{-1/2}(\partial G)\rightarrow H^{1/2}(\partial G)$ is a Fredholm operator with index 0 (see \cite{mclean}) and $S_1: H^{-1/2}(\partial G)\rightarrow H^{1/2}(\partial G)$ is compact.
		We remark here that $S_0$ is even bijective for $n=3$. Hence $S$ is a Fredholm operator with index $0$. Then by the Fredholm alternative we only need to show $S$ is injective.
		That is
		\begin{equation*}
		v_+\big|_{\partial G}=S[\varphi]_+\big|_{\partial G}=0 \quad \text{ with }\quad \varphi\in H^{-1/2}(\partial G)\Longrightarrow\varphi=0,
		\end{equation*}
		where $v_+=S[\varphi]\big|_{\Omega\setminus\overline G}$.
		Observe that we have
		$$
		\begin{cases}
		\Delta v_+=0&\text{ in }\,\,\Omega
		\setminus\overline G,\\
		v_+=0 &\text{ on }\partial\Omega,\\
		v_+=0&\text{ on }\partial G.
		\end{cases}
		$$
		By the uniqueness of this boundary value problem, we have $v_+=0$ in $\overline{\Omega\setminus\overline G}$.
		On the other hand $v_-:=S[\varphi]|_{G}$
		satisfies
		$$
		\begin{cases}
		\Delta v_-=0&\text{ in }G,\\
		v_-=v_+=0&\text{ on }\partial G,
		\end{cases}
		$$
		which implies $v_-=0$ in $\overline G$.
		Then by the jump formula for $\partial_\nu S[\varphi]$ at $\partial G$, we have $\varphi=0$. Here note that the singularity of the kernels of $S$ and $S_0$ are the same at $\partial G$.
	\end{proof}
	
	Next we will prove the denseness of the ranges of the operators $R^{\ast}$ and $W$. Since the proofs are almost the same, we only give the details for $R^\ast$ and just point out some additional things to be concerned for $W$.
	\begin{prop}\label{denseness of Rge(R ast)}
		$R$ is injective. Hence the range of $Rge(R^\ast)\subset H^{-1/2}(\partial G)$ of $R$ is dense. 
	\end{prop}
	\begin{proof}
		It is well known that we have $\overline{Rge(R^\ast)}=N(R)^\perp$, where $N(R)$ is the kernel of $R$. Hence it is enough to prove the injectivity of $R$. To show this let
		$\varphi\in H^{-1/2}(\partial G),\,R\varphi=0$ on $\partial\Omega$. Put $$
		q(x):=\int_{\partial G}\partial_{\nu_x}
		\G(x,y)\varphi(y)\,d\sigma (y),\quad x\in\Omega\setminus\partial G,
		$$
		then $q$ satisfies 
		\begin{equation}\label{uniqueness of q}
		\begin{cases}
		\Delta q=0 &\text{ in }{\mathbb R}^n\setminus\partial G,\\
		q=0  &\text{ on }\partial\Omega,\\
		q(x)=O(|x|^{-(n-1)}),\,\nabla q(x)=O(|x|^{-n}) &\text{ as }|x|\rightarrow\infty.
		\end{cases}
		\end{equation}
		By the uniqueness of the exterior problem (see Chapter 8 of \cite{Mizohata}), we have
		\begin{equation}\label{q zero}
		q=0\,\,\text{in}\,\,{\mathbb R}^n\setminus\Omega
		\end{equation}
		and this implies $\partial_\nu q=0$ on $\partial\Omega$. Here note that we have one order faster decay for $q(x)$ and $\nabla q(x)$, but the order given in \eqref{uniqueness of q} is enough to have \eqref{q zero}. Then by the UCP, we have $q=0$ in ${\mathbb R}^n\setminus\overline G$ and this implies $(\partial_\nu q)_+\big|_{\partial G}=0$. By the continuity of the normal derivative of $q$ at $\partial G$, we have the trace of the normal derivative $(\partial_\nu q)_-\big|_{\partial G}$ taken from $G$ is zero, i.e.$(\partial_\nu q)_-\big|_{\partial G}=0$. Recalling $\Delta q=0$ in $G$, there exists a constant $c$ such that $q=c$ on $\overline G$. Hence by the jump formula, $\varphi=c$, and hence
		$$
		q(x)=c\int_{\partial G}\partial_{\nu_x} \G(x,y)\,d\sigma (y),\quad x\in {\mathbb R}^n\setminus\partial G.
		$$
		This contradicts to $q=0$ in ${\mathbb R}^3\setminus\overline G$ if $c\not=0$. Hence $c$ must be zero and hence $\varphi=0$.
		Thus we have proven the injectivity of $R$.
	\end{proof}	
	\begin{prop}\label{denseness of Rge(W)}
		Consider the layer potential operator $W$ defined by \eqref{double-layer potential} as an operator $W:H^{1/2}(\partial\Omega)\rightarrow H^{1/2}(\partial\widetilde M(z))$. Then	the range $Rge(W)\subset H^{1/2}(\partial\widetilde M(z))$ of $W$ is dense.
	\end{prop}
	\begin{proof}
		As mentioned before the proof is quite similar to the proof of Proposition \ref{denseness of Rge(R ast)}. Hence we will only point out some additional things necessary for proving the denseness of $Rge(W)$. Denote $X:=H^{-1/2}(\partial\widetilde M(z))$, $Y:=H^{-1/2}(\partial\Omega)$ and their dual spaces by $X^\ast=H^{1/2}(\partial\widetilde M(z))$ and $Y^\ast=H^{1/2}(\partial\Omega)$, respectively.
		Also let $J_X: X\rightarrow X^\ast$ and $J_Y: Y\rightarrow Y^\ast$ be the isometric isomorphisms.  Let $W^{\ast}:X^\ast\rightarrow Y^\ast$ and $W^{(\ast)}: X\rightarrow Y$ be the adjoint operator and dual operator of $W$. Then since $W^{\ast}=J_Y W^{(\ast)}J_X^{-1}$, it is enough to prove $W^{(\ast)}$ is injective. By a direct computation, $W^{(\ast)}$ is given as
		$$
		W^{(\ast)}[\psi](y)=\int_{\partial G}\partial_{\nu_y}\G(x,y)\psi(x)\,d\sigma(x),\,\,y\in\partial\Omega
		$$
		for any $\psi\in X$.
		Further, by the denseness of $X^\ast\subset X$ and the continuity of $W^{(\ast)}$, it is enough to prove that if $\psi\in X^\ast$,  $W^{(\ast)}[\psi](y)=0,\,y\in\partial\Omega$, then $\psi=0$. The rest of the proof is almost the same as that of Proposition \ref{denseness of Rge(R ast)}.
	\end{proof}

	\vspace{1mm}
	
	\noindent \textbf{Acknowledgment.}
	%The second author appreciate Roland Potthast for having a useful discussion on the proof of the identity corresponding to \eqref{manipulation3} in the proof of the duality between the RT and NRT for the inverse scattering problem \cite{Nakamura}. 
	The second author appreciate Professor Guanghui Hu at Beijing Computational Science Research Center for several useful discussions. 
	For the research funds, the first author, the second author and the fourth author acknowledge their supports from the Ministry of Science and Technology Taiwan, under the Columbus Program: MOST-109-2636-M-009-006, the Grant-in-Aid for Scientific Research 19K03554 and 15H05740 of JSPS and the National Natural Science Foundation of China (Nos.11671082, 11971104), respectively.


\begin{thebibliography}{99}
		
		\bibitem{Alessandrini}
		G. Alessandrini, E. Beretta, E. Rosset and S. Vessella, Optimal stability for inverse elliptic boundary value problems with unknown boundaries, Ann. Scuola Norm. Sup. Pisa Cl. Sci., 29 (2000), 755--806, 2000.
		
		\bibitem{Ammari}
		H. Ammari and H. Kang, "Polarization and Moment Tensors, With applications to inverse problems and effective medium theory," Springer-Verlag, Berlin 2007.
		
		\bibitem{Ang}
		D.D. Ang, D.D. Trong and M. Yamamoto, Unique continuation and identification of boundary of an elastic body, J. Inverse Ill-Posed Probl., 3 (1996), 417--428, 1996.
		
		\bibitem{Bonnet}
		M. Bonnet, Inverse acoustic scattering by small-obstacle expansion of a misfit function, Inverse Problems, 24 (2008), 035022.
		
		\bibitem{Bourgeois}
		L. Bourgeois and J. Dard\'e, A quasi-reversibility approach to solve the inverse obstacle problem, Inverse Probl. Imaging, 4 (2010), 351--377.
		
		\bibitem{Burger}
		M. Burger, A level set method for inverse problems, Inverse Problems, 17 (2001), 1327--1355.
		
		\bibitem{Zou} 
		T. Chow, K. Ito and J. Zou, A direct sampling method for electrical impedance tomography, Inverse Problems, 30 (2014), 095003.
		
		\bibitem{colton-kress}
		D. Colton and R. Kress, 
		"Inverse Acoustic and Electromagnetic Scattering Theory," 3rd edition, Springer-Verlag, Berlin 2013.
		
		\bibitem{Friedman}
		A. Friedman and V. Isakov, On the uniqueness in the inverse conductivity problem with one measurement, Indiana University Mathematics Journal, 38 (1989), 563--579. 
		
		\bibitem{Higashimori}
		N. Higashimori, A conditional stability estimate for determining a cavity in an elastic material, Proc. Japan Acad. Ser. A Math. Sci., 78 (2002), 15--17.
		
		\bibitem{HNS}
		N. Honda, G. Nakamura and M. Sini, Analytic extension and reconstruction of obstacles from few measurements for elliptic second order operators, Math. Ann., 355 (2013), 401-427.
		
		\bibitem{HPNS}
		N. Honda, G. Nakamura, R. Potthast and M. Sini, The no-response approach and its relation to non-iterative methods for the inverse scattering, Annali di Matematica, 187 (2008), 7--37.
		
		\bibitem{Ikehata convex} 
		M. Ikehata, Enclosing a polygonal cavity in a two-dimensional bounded domain from Cauchy data, Inverse Problems, 15 (1999), 1231--1241. 
		
		
		\bibitem{ikehata}
		M. Ikehata and T. Ohe, A numerical method for finding the convex hull of polygonal cavities using the enclosure method, Inverse Problems, 18 (2002), 111--124.
		
		\bibitem{Isakov}
		V. Isakov, "Inverse Problems for Partial Differential Equations," 3rd edition, Springer Verlag, Berlin 2017. 
		
		\bibitem{Kress}
		R. Kress and W. Rundell, Nonlinear integral equations and the iterative solution for an inverse boundary value problem, Inverse Problems, 21 (2005), 1207--1223.
		
		\bibitem{Luke}
		D. R. Luke and R. Potthast, The no response test-a sampling method for inverse scattering problems, SIAM J. Appl. Math., 63 (2003), 1292--1312.
		
		\bibitem{mclean}
		W. McLean, "Strongly Elliptic Systems and Boundary Integral Equations,"  Cambridge University Press, Cambridge 2000.
		
		\bibitem{Mizohata}
		S. Mizohata, "The Theory of Partial Differential Equations," Cambridge University Press, Cambridge 1973.
		
		\bibitem{Morassi}
		A. Morrassi and E. Rosset, Stable determination of cavities in elastic bodies, Inverse Problems, 20 (2004), 453--480.
		
		
		
		\bibitem{Nakamura} G. Nakamura and R. Potthast, "Inverse Modelling," IOP Publishing, Bristol 2015. 
		
		\bibitem{Potthast} R. Potthast, On the convergence of the no response test, SIAM J. Math. Anal., 38 (2007), 1808--1824.
		
		\bibitem{PS}
		R. Potthast and M. Sini, The no response test for the reconstruction of polyhedral objects in electromagnetics, J. Comput. Appl. Math., 234 (2010), 1739-1746.
		
		\bibitem{KPS}
		R. Potthast, J. Sylvester and S. Kusiak, A 'range test' for determining scatteres with unknown physical properties, Inverse Problems, 19 (2003), 533--548.
		
		\bibitem{ZP}
	    Q. Zia and R. Potthast, The range test and the no response test for Oseen problems: Theoretical foundation, J. Comput. Appl. Math., 304 (2016), 201--211.
		
		
		
	\end{thebibliography}
\end{document}